\documentclass[a4paper,10pt]{amsart}
\usepackage[top=0.8in, bottom=0.8in, left=1.25in, right=1.25in]{geometry}
\usepackage{amsmath}
\usepackage{amsfonts}
\usepackage{amssymb}
\usepackage{graphicx}
\usepackage{mathrsfs}
\usepackage{pb-diagram}
\usepackage{epstopdf}
\usepackage{CJK}

\usepackage{amscd,amsthm,curves,enumerate,latexsym,multibox}
\usepackage{xypic}
\usepackage[all]{xy}
\usepackage{epstopdf}
\newtheorem{theorem}{Theorem}[section]
\newtheorem{lemma}[theorem]{Lemma}
\newtheorem{proposition}[theorem]{Proposition}
\newtheorem{corollary}[theorem]{Corollary}
\theoremstyle{definition}
\newtheorem{definition}[theorem]{Definition}
\newtheorem{remark}[theorem]{Remark}

\newtheorem{notation}[theorem]{Notation}

\begin{document}
\title[Remarks on Mirror Symmetry of DT theory for Calabi-Yau 4-folds]{Remarks on Mirror Symmetry of \\ Donaldson-Thomas theory for Calabi-Yau 4-folds}
\author{Yalong Cao}
\address{The Institute of Mathematical Sciences and Department of Mathematics, The Chinese University of Hong Kong, Shatin, Hong Kong}
\email{ylcao@math.cuhk.edu.hk}

\author{Naichung Conan Leung}
\address{The Institute of Mathematical Sciences and Department of Mathematics, The Chinese University of Hong Kong, Shatin, Hong Kong}
\email{leung@math.cuhk.edu.hk}

\maketitle
\begin{abstract}
Motivated by Strominger-Yau-Zaslow's mirror symmetry proposal and Kontsevich's homological mirror symmetry conjecture,
we study mirror phenomena (in A-model) of certain results from Donaldson-Thomas theory for Calabi-Yau 4-folds.
\end{abstract}


\section{Introduction}
Mirror symmetry is a duality between symplectic geometry (A-model) and complex geometry (B-model) for Calabi-Yau manifolds \cite{yau}.
In the B-model, Donaldson-Thomas invariants \cite{dt, th, js, ks} count holomorphic bundles (or coherent sheaves) on Calabi-Yau 3-folds. Borisov and Joyce \cite{bj} and the authors \cite{cao, caoleung, caoleung3, caoleung2} studied their extensions to Calabi-Yau 4-folds (abbrev. $CY_{4}$).

The purpose of this note is to study certain corresponding mirror phenomena in the A-model for $CY_{4}$, mainly motivated by Strominger-Yau-Zaslow's geometric mirror symmetry proposal \cite{syz, lyz}, Kontsevich's homological mirror symmetry conjecture \cite{kont}
and Thomas' paper on $CY_{3}$ \cite{thomas}. In particular, we study calibrated geometry \cite{harveylawson} for $CY_{4}$ and point out corresponding structures in $DT_{4}$ theory (B-model). We continue Thomas' table \cite{thomas} as follows.

\[%
\begin{tabular}
[c]{|l|l|l|}\hline
$\textrm{Topological twists} %
\begin{array}
[c]{l}%
\,\\
\,
\end{array}
$ & \quad \quad \quad \quad  B-model & \quad \quad \quad \quad  A-model \\\hline
$\textrm{Calabi-Yau 4-folds} %
\begin{array}
[c]{l}%
\,\\
\,
\end{array}
$ & \quad \quad \quad \quad \quad  $\check{X}$   & \quad \quad \quad \quad \quad  $X$   \\\hline
$\textrm{Complex/ Symplectic}%
\begin{array}
[c]{l}%
\,\\
\,
\end{array}
$  & \quad $\Omega=\Omega_{\check{X}}\in H^{4,0}(\check{X})$ & \quad $\omega=\omega_{X}\in H^{1,1}(X)$  \\
 structures & \quad $\omega=\omega_{\check{X}}\in H^{1,1}(\check{X})$  & \quad $\Omega=\Omega_{X}\in H^{4,0}(X)$ \\\hline
$\textrm{Geometric objects}%
\begin{array}
[c]{l}%
\,\\
\,
\end{array}
$ & Connections on a vector bundle & Submanifolds in class $[L]\in H^{4}(X)$ \\
& \quad \quad \quad $E\rightarrow \check{X}$  & \quad  with connections on $E\rightarrow L$  \\\hline
$\textrm{Star operators}%
\begin{array}
[c]{l}%
\,\\
\,
\end{array}
$ & Choose a metric $h_{E}$ on $E$  & Choose a metric $h$ on $E$, $g_{L}=g_{X}|_{L}$   \\
& $*_{4}\triangleq(\Omega\lrcorner)\circ*_{h_{E}}\circlearrowleft\Omega^{0,\bullet}(\check{X},EndE)$ &  $\quad *_{g_{L}}\circlearrowleft \Omega^{\bullet}(L)$, \textrm{  } $*_{h}\circlearrowleft\Omega^{\bullet}(L,\mathfrak{g}_{E})$    \\\hline
$\textrm{Energy functionals}%
\begin{array}
[c]{l}%
\,\\
\,
\end{array}
$ & \quad \quad $\int_{\check{X}}|F^{0,2}|_{h_{E}}^{2}dvol$ & \quad \quad $\int_{L}(|F|_{h}^{2}+|\omega|_{L}|^{2})dvol_{g_{L}}$  \\\hline
$\textrm{Energy minimizers}%
\begin{array}
[c]{l}%
\,\\
\,
\end{array}
$ & \quad $F^{0,2}+*_{4}F^{0,2}=0$ & \quad $\omega|_{L}+*_{g_{L}}(\omega|_{L})=0$, \textrm{  }$F^{+}=0$ \\
& Complex ASD connections &  ASD submfds with ASD bundles  \\\hline

$ \textrm{Reductions}%
\begin{array}
[c]{l}%
\,\\
\,
\end{array}
$ &  If $ch_{2}(E)\in\textrm{Ker}(\wedge[\Omega]) \cap H^{4}(\check{X})$ & \quad If $[L]\in\textrm{Ker}(\wedge[\omega^{2}])\cap H^{4}(X)$   \\
& \quad $F^{0,2}_{+}=0\Rightarrow F^{0,2}=0$ & \quad \quad $(\omega|_{L})^{+}=0\Rightarrow \omega|_{L}=0$  \\\hline
$\textrm{Moment maps}%
\begin{array}
[c]{l}%
\,\\
\,
\end{array}
$ & \quad\quad\quad\quad $F\wedge \omega^{3}$  &  \quad \quad\quad \quad Im$(\Omega)|_{L}$ \\\hline
\end{tabular}
\]
\\

The ASD submanifolds mentioned in the above table (see section 2) are corresponding mirror objects of complex ASD connections on $CY_{4}$.
To continue the discussion, let us first fix the following notation.
\begin{notation}\label{notation}
Unless specified otherwise, we denote \\
(1) $X$ to be a Calabi-Yau 4-fold (compact or convex at infinity with $c_{1}(X)=0$); \\
(2) $L$ to be a compact relatively spin Lagrangian submanifold in $X$ with zero Maslov index.
\end{notation}
In the definition of $DT_{4}$ invariants (B-model), Brav-Bussi-Joyce's local Darboux theorem \cite{bbj} (see Theorem \ref{B-side local Darboux thm}) for moduli spaces of simple sheaves on $CY_{4}$ is an important ingredient, which says for any simple sheaf $\mathcal{F}$, we could choose a local Kuranishi map
\begin{equation}\kappa: Ext^{1}(\mathcal{F},\mathcal{F})\rightarrow Ext^{2}(\mathcal{F},\mathcal{F})  \nonumber \end{equation}
such that
\begin{equation}\int_{X}Tr(\kappa\cup\kappa)\cup\Omega_{X}=0. \nonumber \end{equation}
We are interested in the corresponding mirror result in the A-model.
In fact, the analog of the above Kuranishi map in A-model is
\begin{equation}\kappa: H^{1}(L;\Lambda_{0,nov}^{+})\rightarrow H^{2}(L;\Lambda_{0,nov}^{+}),  \nonumber \end{equation}
\begin{equation}\kappa(x)\triangleq\sum_{k=0}^{\infty}m_{k}(x^{\otimes k}),  \nonumber \end{equation}
where $\{m_{k}\}_{k\geq0}$ is the $A_{\infty}$-algebra structure on $H^{*}(L;\Lambda_{0,nov}^{+})$ defined by Fukaya \cite{fukaya}.
\begin{theorem}\label{thm 1}(Theorem \ref{A-side local Darboux thm}) \\
Let $L\subseteq X$ be a Lagrangian submanifold in a $CY_{4}$. Then
\begin{equation}Q(\kappa,\kappa)=\textrm{const}, \nonumber\end{equation}
where $Q$ is the Poincar\'{e} pairing on $H^{2}(L;\Lambda_{0,nov}^{+})$.
\end{theorem}
This theorem follows from a combination of a general result for cyclic $A_{\infty}$-algebras (see Lemma \ref{lemma on cyclic str}) and the existence of a cyclic $A_{\infty}$-structure on $H^{*}(L;\Lambda_{0,nov}^{+})$ \cite{fukaya}.

If $L$ is an \emph{unobstructed} Lagrangian, i.e. there exists $b\in H^{1}(L;\Lambda_{0,nov}^{+})$ such that $\kappa(b)=0$,
one can define the twisted $A_{\infty}$-algebra $(H^{*}(L;\Lambda_{0,nov}^{+}),m_{k}^{b})$ with
\begin{equation}m_{k}^{b}(x_{1},\cdot\cdot\cdot,x_{k})=
\sum_{n\geq k}\sum m_{n}(b,\cdot\cdot\cdot,b,x_{1},b,\cdot\cdot\cdot,b,\cdot\cdot\cdot,x_{k},b,\cdot\cdot\cdot,b),  \nonumber \end{equation}
where the first summation is taken over all such expressions.
The corresponding Kuranishi map
\begin{equation}\kappa^{b}:H^{1}(L;\Lambda_{0,nov}^{+})\rightarrow H^{2}(L;\Lambda_{0,nov}^{+}), \quad \kappa^{b}(x)=\sum_{k=0}^{\infty}m^{b}_{k}(x^{\otimes k})  \nonumber \end{equation}
is similarly defined.

$(H^{*}(L;\Lambda_{0,nov}^{+}),m_{k}^{b})$ is a cyclic $A_{\infty}$-algebra with $m_{0}^{b}(1)=0$ provided $(H^{*}(L;\Lambda_{0,nov}^{+}),m_{k})$ is a cyclic $A_{\infty}$-algebra (see also \cite{fukaya}). As a corollary of the above theorem, we get an unobstructedness result for moduli spaces of Maurer-Cartan elements, i.e. if the space of bounding cochains $b$'s is nonempty, then it is the whole $H^{1}(L;\Lambda_{0,nov}^{+})$.
\begin{theorem}(Theorem \ref{cor on kuranishi map}) \\
Let $L\subseteq X$ be a definite \footnote{i.e. the intersection form on $H^{2}(L,\mathbb{R})$ is definite.} and unobstructed Lagrangian submanifold in a $CY_{4}$. Then \\
(1) $\kappa\equiv0$;
(2) for any $b\in H^{1}(L;\Lambda_{0,nov}^{+})$, $\kappa^{b}\equiv0$.
\end{theorem}
${}$ \\
The outline of this note is as follows: In section 2, we introduce the Harvey-Lawson (anti)-self-dual submanifolds in $CY_{4}$ and study their basic properties. We also discuss an orientability problem for moduli spaces of special Lagrangian submanifolds. In section 3, we study FOOO's Lagrangian Floer theory on $CY_{4}$ and point out corresponding structures in the B-side. In the final section, we recall basic facts
in $DT_{4}$ theory (B-side story). \\
${}$ \\
\textbf{Acknowledgement}: The first author would like to thank Garrett Alston, Yin Li and Junwu Tu for many helpful discussions. Special thanks to Alston for teaching him Lagrangian Floer theory. He also expresses his gratitude to Simon Donaldson for pointing out orientability problems in calibrated geometry during a visit to the Simons Center. The work of the second author was substantially supported by grants from the Research Grants Council of the Hong Kong Special Administrative Region, China (Project No. CUHK401411 and CUHK14302714).

\section{Mirror aspects of $DT_{4}$ theory in calibrated geometry}

\subsection{Harvey-Lawson (anti-)self-duals in eight manifolds}
We recall that the mirror of holomorphic bundles (resp. HYM bundles) on Calabi-Yau manifolds are Lagrangian submanifolds (resp. special Lagrangian submanifolds) coupled with flat bundles. In complex 4-dimension, under SYZ mirror symmetry \cite{syz} \cite{lyz}, solutions to
the $DT_{4}$ equations
\begin{equation} \left\{ \begin{array}{l}
  F^{0,2}_{+}=0 \\ F\wedge\omega^{3}=0 ,    
\end{array}\right.
\nonumber \end{equation}
become special Harvey-Lawson ASD submanifolds coupled with ASD bundles as described below.
\begin{definition}\label{HL asd submfd}
Given an almost Hermitian eight manifold\footnote{It is an almost complex manifold with a Hermitian metric.} $(X,g,J,\omega)$, an oriented four-dimensional submanifold $L$ is a Harvey-Lawson
anti-self-dual submanifold if
\begin{equation}(\omega|_{L})^{+}\triangleq\frac{1}{2}(\omega|_{L}+*(\omega|_{L}))=0\in\Omega^{2}_{+}(L),
\nonumber \end{equation}
where $*$ is the Hodge-star operator on $L$ for the induced metric $g|_{L}$.

When $(X,g,J,\omega)$ is a $CY_{4}$ with holomorphic volume form $\Omega$, a Harvey-Lawson ASD submanifold $L$ is special if it
satisfies
\begin{equation}Im(\Omega)|_{L}=0. \nonumber \end{equation}
\end{definition}
\begin{remark}
This notion was introduced by Harvey-Lawson \cite{harveylawson} for submanifolds in $\mathbb{C}^{4}$ . They also showed
special ASD submanifolds are exactly the same as Cayley submanifolds with respect to the Cayley 4-form $Re(\Omega)-\frac{1}{2}\omega^{2}$.
\end{remark}
If $d\omega=0$, i.e. $X$ is almost K\"{a}hler, Lagrangian submanifolds are Harvey-Lawson ASD's. A converse statement is given by
\begin{proposition}
Let $(X,g,J,\omega)$ be an almost K\"{a}hler eight manifold, $L$ be a closed Harvey-Lawson ASD submanifold such that $[(\omega|_{L})^{2}]=0\in H^{4}(L)$. Then $L$ is a Lagrangian submanifold.
\end{proposition}
\begin{proof}
By \cite{dk}, we have an identity
\begin{equation}(\omega|_{L})^{2}=(|(\omega|_{L})^{+}|^{2}-|(\omega|_{L})^{-}|^{2})dvol_{L}. \nonumber \end{equation}
From the Stokes theorem and $[(\omega|_{L})^{2}]=0\in H^{4}(L)$, we obtain
\begin{equation}0=\int_{L}(\omega|_{L})^{2}=-\int_{L}|(\omega|_{L})^{-}|^{2}dvol_{L}. \nonumber \end{equation}
Finally, we use the energy identity
\begin{equation}\int_{L}|(\omega|_{L})|^{2}=\int_{L}(|(\omega|_{L})^{+}|^{2}+|(\omega|_{L})^{-}|^{2})dvol_{L} \nonumber \end{equation}
to get the conclusion.
\end{proof}
\begin{remark}
Any Harvey-Lawson ASD with $b_{2}=0$ is a Lagrangian submanifold.
\end{remark}
\begin{remark}
Besides Lagrangian submanifolds, half-dimensional almost K\"{a}hler submanifolds $L$'s are also examples of Harvey-Lawson self-dual's,
because $\omega|_{L}$ is the almost K\"{a}hler form of $L$ which is a self-dual two form on $(L,g|_{L})$ \cite{dk}.
This shows Harvey-Lawson ASD's could have obstructed deformations in general.
\end{remark}
\begin{remark}(Harvey-Lawson ASD's under geometric flows)  \\
Lagrangian submanifolds in K\"ahler-Einstein manifolds (e.g. Calabi-Yau manifolds) are preserved under the mean curvature flow whose stationary points are minimal Lagrangians (they are special Lagrangians in the Calabi-Yau case). For general K\"ahler manifolds, one need to couple the K\"ahler-Ricci flow with the mean curvature flow to preserve the Lagrangian condition \cite{smoczyk}.

Lotay and Pacini \cite{lotay} extended the above result to totally real submanifolds in almost K\"{a}hler manifolds by coupling the symplectic curvature flow \cite{stian} (a generalization of K\"ahler-Ricci flow) with the Maslov flow (a generalization of MCF).

In a K\"ahler-Einstein manifold, Maslov flow preserves the pull-back of the K\"{a}hler form to any totally real submanifold \footnote{Totally real is an open condition in the Grassmannian of all $4$-planes inside $\mathbb{C}^{4}$.}. If we use a fixed metric instead of the induced metric, the flow preserves totally real Harvey-Lawson ASD's.
\end{remark}

%

\subsection{Orientations for moduli spaces of special Lagrangians in Calabi-Yau manifolds}
In this section, we study the mirror of the orientability result for moduli spaces of sheaves on Calabi-Yau manifolds \cite{caoleung3}. We first recall the moment map approach to the moduli space of (special-)Lagrangians in Calabi-Yau manifolds, which is the beautiful work of Donaldson \cite{d10} and Hitchin \cite{hitchin2}.

Let $L$ be a closed $n$-manifold with a fixed volume form $dvol_{L}$, and $X$ be a Calabi-Yau $n$-fold with a K\"{a}hler form $\omega$ and a holomorphic volume form $\Omega$. We consider the space $\textrm{Map}_{0}(L,X)$
of smooth maps $f$'s with $f^{*}[\omega]=0\in H^{2}(L)$, and a symplectic form $\varphi$ on it defined by
\begin{equation}\varphi|_{(f)}: \Omega^{0}(L,f^{*}TX)\otimes\Omega^{0}(L,f^{*}TX)\rightarrow \mathbb{R},   \nonumber \end{equation}
\begin{equation}\varphi|_{(f)}(v_{1},v_{2})=\int_{L}\omega(v_{1},v_{2})dvol_{L}.  \nonumber \end{equation}
The group $\mathcal{G}=\textrm{Diff}_{dvol_{L}}(L)$ of volume-preserving diffeomorphisms acts on $\textrm{Map}_{0}(L,X)$ preserving the symplectic form $\varphi$. The zero loci of the corresponding moment map consists precisely of those maps $f$'s with $f^{*}(\omega)=0$.

The complex structure on $X$ induces a complex structure on $\textrm{Map}_{0}(L,X)$, and the subspace
\begin{equation}S=\{f\in\textrm{Map}_{0}(L,X)\textrm{ } | \textrm{ } f^{*}(\Omega)=dvol_{L} \}   \nonumber \end{equation}
is a complex submanifold of $\textrm{Map}_{0}(L,X)$ consisting of immersions.
We take the subgroup $\mathcal{G}_{0}\subseteq\mathcal{G}$ to be the kernel of the Calabi map \cite{d10}, \cite{banyaga}. The symplectic quotient $\mathcal{M}^{c}\triangleq S//\mathcal{G}_{0}$ is a Lagrangian torus bundle (with fiber $H_{1}(L,\mathbb{R})/H_{1}(L,\mathbb{Z})$) over the moduli space $\mathcal{M}$ of (immersed) special Lagrangian submanifolds. $\mathcal{M}$ has an integral affine structure by the Arnold-Liouville theorem.

We define vector bundles $E^{*}=(S\times H^{*}(L,\mathbb{C}))//\mathcal{G}_{0}$ over $\mathcal{M}^{c}$, and form the determinant complex line bundle
$\mathcal{L}=det(E^{*})\rightarrow \mathcal{M}^{c}$.
\begin{proposition}\label{exi of ori on lag} ${}$ \\
(1) if $n$ is even, $c_{1}(\mathcal{L})=0$ provided that $H_{1}(\mathcal{M}^{c},\mathbb{Z})$ has no 2-torsion elements, \\
(2) if $n$ is odd, $\mathcal{L}$ has a square root.
\end{proposition}
\begin{proof}
(1) As $\mathcal{G}_{0}$ preserves the volume form on $L$, the Poincar\'{e} pairing on $H^{*}(L,\mathbb{C})$ induces an isomorphism $\mathcal{L}\cong \mathcal{L}^{*}$ between complex line bundles when $n$ is even. Since $H^{2}(\mathcal{M}^{c},\mathbb{Z})$ has no 2-torsion elements, $2c_{1}(\mathcal{L})=0\Rightarrow c_{1}(\mathcal{L})=0$.

(2) If $n$ is odd, $det(E^{odd})$ gives a square root of $\mathcal{L}$ by Poincar\'{e} duality.
\end{proof}

\section{Mirror aspects of $DT_{4}$ theory in Lagrangian Floer theory}

\subsection{Mirror results}
In this section, we study mirror phenomena of $DT_{4}$ theory from the perspective of Lagrangian Floer theory.
Lagrangian Floer cohomology $HF^{*}(L)$, introduced by Fukaya, Oh, Ohta and Ono \cite{fooo}, is defined in terms of counting holomorphic disks
bounding Lagrangian submanifold $L$.
Given a Calabi-Yau mirror pair $(X,\check{X})$, there should exist a correspondence
\begin{equation}Ext^{*}_{\check{X}}(\mathcal{F},\mathcal{F})\leftrightarrow HF^{*}(L)   \nonumber\end{equation}
under mirror symmetry \cite{kont}. In good cases, $HF^{*}(L)\cong H^{*}(L)$ and Serre duality pairing in the B-model would be mirror to the Poincar\'{e} pairing in the A-model.
On $CY_{3}$, moduli spaces of $\mathcal{F}$ (resp. $(L,b)$\footnote{$b$ is a bounding cochain which helps to define $HF^{*}(L)$ (see Fukaya \cite{fukaya1}).}) are locally critical points of holomorphic functions \cite{bbj}, \cite{js} (resp. \cite{fukaya1}).
On $CY_{4}$, we have local 'Darboux models' for moduli spaces of simple sheaves (Theorem \ref{B-side local Darboux thm}), we expect
a similar structure in the A-model.

To state the result, we first introduce the Kuranishi map
\begin{equation}\kappa: H^{1}(L;\Lambda_{0,nov}^{+})\rightarrow H^{2}(L;\Lambda_{0,nov}^{+}),  \nonumber \end{equation}
\begin{equation}\kappa(x)\triangleq\sum_{k=0}^{\infty}m_{k}(x^{\otimes k}),  \nonumber \end{equation}
where $\{m_{k}\}_{k\geq0}$ is the $A_{\infty}$-algebra structure on $H^{*}(L;\Lambda_{0,nov}^{+})$ defined by Fukaya \cite{fukaya}.
\begin{theorem}\label{A-side local Darboux thm}
Let $L\subseteq X$ be a Lagrangian submanifold in a $CY_{4}$. Then
\begin{equation}Q(\kappa,\kappa)=\textrm{const}, \nonumber\end{equation}
where $Q$ is the Poincar\'{e} pairing on $H^{2}(L;\Lambda_{0,nov}^{+})$.
\end{theorem}
In fact, this result follows from a combination of the existence of a cyclic $A_{\infty}$-structure on $H^{*}(L;\Lambda_{0,nov}^{+})$ due to Fukaya \cite{fukaya} (see Theorem \ref{cyclic A inf str}) and the following lemma on cyclic $A_{\infty}$-algebras.
\begin{lemma}\label{lemma on cyclic str}
Given a cyclic $A_{\infty}$-algebra $(A,m_{k},Q)$, for any $k\geq0$ and $x\in A^{1}$, we have
\begin{equation}\sum_{k_{1}+k_{2}=k+1} Q(m_{k_{1}}(x^{\otimes k_{1}}),m_{k_{2}}(x^{\otimes k_{2}}))=0.  \nonumber\end{equation}
In particular, $Q(\kappa,\kappa)=Q(m_{0}(1),m_{0}(1))$, where $\kappa: A^{1}\rightarrow A^{2}$, $\kappa(x)\triangleq\sum_{k=0}^{\infty}m_{k}(x^{\otimes k})$
is the Kuranishi map of $(A,m_{k})$.
\end{lemma}
\begin{proof}
From Definition \ref{def of cyclic str}, given $k_{1},k_{2}\geq0$ with $k_{1}+k_{2}\geq1$, we have
\begin{equation}Q(m_{k_{1}}(x^{\otimes k_{1}}),m_{k_{2}}(x^{\otimes k_{2}}))=Q(m_{k_{1}}(x^{\otimes r},m_{k_{2}}(x^{\otimes k_{2}}),x^{\otimes t}),x), \nonumber\end{equation}
for $r,t\geq0$ with $r+t+1=k_{1}$. We fix $k_{1}+k_{2}=k+1\geq1$, then
\begin{equation}(\frac{k+1}{2})\sum_{k_{1}+k_{2}=k+1} Q(m_{k_{1}}(x^{\otimes k_{1}}),m_{k_{2}}(x^{\otimes k_{2}}))=\sum_{k_{1}+k_{2}=k+1}\sum_{r+t+1=k_{1}} Q(m_{k_{1}}(x^{\otimes r},m_{k_{2}}(x^{\otimes k_{2}}),x^{\otimes t}),x), \nonumber\end{equation}
which is zero by the $A_{\infty}$-relation.
\end{proof}
On $CY_{4}$, local 'Darboux models' for moduli spaces of stable sheaves (Theorem \ref{B-side local Darboux thm}) have an application to the unobstructedness of these moduli spaces (Corollary \ref{unobs in B side}). We expect a similar result for moduli spaces of Maurer-Cartan elements of $A_{\infty}$-algebras $H^{*}(L;\Lambda_{0,nov}^{+})$'s (one could work with non-archemedian geometry to make sense the moduli space as done in \cite{fukaya}).

By SYZ mirror symmetry proposal \cite{syz}, \cite{lyz} and Kontsevich's HMS conjecture \cite{kont}, a sheaf (with $Ext^{*}$ group) in the B-model is mirror to a Lagrangian (we take the flat bundle to be trivial for simplicity) with a bounding cochain (i.e. a Maurer-Cartan element which helps to define $HF^{*}$) in the A-model. As a corollary of Theorem \ref{A-side local Darboux thm}, the unobstructedness result in the A-model should be stated as follows.

We start with an \emph{unobstructed} Lagrangian \footnote{i.e. there exists $b\in H^{1}(L;\Lambda_{0,nov}^{+})$ such that $\kappa(b)=0$.}, define the twisted $A_{\infty}$-algebra $(H^{*}(L;\Lambda_{0,nov}^{+}),m_{k}^{b})$
\begin{equation}m_{k}^{b}(x_{1},\cdot\cdot\cdot,x_{k})=
\sum_{n\geq k}\sum m_{n}(b,\cdot\cdot\cdot,b,x_{1},b,\cdot\cdot\cdot,b,\cdot\cdot\cdot,x_{k},b,\cdot\cdot\cdot,b),  \nonumber \end{equation}
where the first summation is taken over all such expressions.
The corresponding Kuranishi map
\begin{equation}\kappa^{b}:H^{1}(L;\Lambda_{0,nov}^{+})\rightarrow H^{2}(L;\Lambda_{0,nov}^{+}), \quad \kappa^{b}(x)=\sum_{k=0}^{\infty}m^{b}_{k}(x^{\otimes k})  \nonumber \end{equation}
is similarly defined.

$(H^{*}(L;\Lambda_{0,nov}^{+}),m_{k}^{b})$ is a cyclic $A_{\infty}$-algebra with $m_{0}^{b}(1)=0$ provided that $(H^{*}(L;\Lambda_{0,nov}^{+}),m_{k})$ is a cyclic $A_{\infty}$-algebra \cite{fukaya}. The unobstructedness result says if the space of bounding cochains $b$'s is nonempty, then it is the whole $H^{1}(L;\Lambda_{0,nov}^{+})$, i.e.
\begin{theorem}\label{cor on kuranishi map}
Let $L\subseteq X$ be a definite\footnote{i.e. the intersection form on $H^{2}(L,\mathbb{R})$ is definite.} and unobstructed Lagrangian in a $CY_{4}$. Then \\
(1) $\kappa\equiv0$; (2) for any $b\in H^{1}(L;\Lambda_{0,nov}^{+})$, $\kappa^{b}\equiv0$.
\end{theorem}
\begin{proof}
(1) Since $L$ is unobstructed, there exists $b$ with $\kappa(b)=\sum_{k=0}^{\infty}m_{k}(b^{\otimes k})=0$. By Theorem \ref{A-side local Darboux thm}, we have
\begin{equation}Q(\kappa,\kappa)=Q(m_{0}(1),m_{0}(1))=Q(\kappa(b),\kappa(b))=0. \nonumber \end{equation}
The definite quadratic form $Q$ on $H^{2}(L,\mathbb{R})$ gives $\kappa\equiv0$, i.e. any element $b\in H^{1}(L;\Lambda_{0,nov}^{+})$ is a bounding cochain.

(2) We define the twisted $A_{\infty}$-algebra $(H^{*}(L;\Lambda_{0,nov}^{+}),m_{k}^{b})$ which is still cyclic \cite{fukaya}. As $m_{0}^{b}(1)=\kappa(b)=0$, we apply Lemma \ref{lemma on cyclic str} to $(H^{*}(L;\Lambda_{0,nov}^{+}),m_{k}^{b})$ and obtain $Q(\kappa^{b},\kappa^{b})=0$. By the definite quadratic form on $H^{2}(L,\mathbb{R})$, we have $\kappa^{b}\equiv0$.
\end{proof}
\begin{remark}
By Donaldson's renowned theorem \cite{d}, \cite{d1}, definite intersection forms on closed smooth 4-manifolds are diagonalizable over integers.   \\
\end{remark}

On some particular type of $CY_{4}$, say $K_{Y}$, where $Y$ is a compact Fano 3-fold, local Kuranishi maps for deformations of (compactly supported) stable sheaves have more refined structures than local 'Darboux models' in Theorem \ref{B-side local Darboux thm} (see Lemma 6.4 \cite{caoleung}).
The refined structure is deduced from the cyclic completion structure on $Ext^{*}(\iota_{*}\mathcal{F},\iota_{*}\mathcal{F})$ \cite{segal}. In general, on the canonical bundle $K_{Y}$ of a compact Fano $n$-fold $Y$, for any
coherent sheaf $\mathcal{F}$, we have
\begin{equation}\label{equ 2}Ext^{*}_{K_{Y}}(\iota_{*}\mathcal{F},\iota_{*}\mathcal{F})\cong Ext^{*}_{Y}(\mathcal{F},\mathcal{F})\oplus Ext^{n+1-*}_{Y}(\mathcal{F},\mathcal{F})^{*}.  \end{equation}
We are interested in its mirror analog in Lagrangian Floer theory (A-model),
and take $Y=\mathbb{P}^{n}$ as an example, whose mirror is given by a superpotential \cite{kont2}, \cite{horiiqbalvafa}
\begin{equation}W=\sum_{i=1}^{n}z_{i}+q(\prod_{i=1}^{n}z_{i})^{-1}: (\mathbb{C}^{*})^{n}\rightarrow \mathbb{C}. \nonumber \end{equation}
Kontsevich's HMS conjecture \cite{kont} predicts an equivalence\footnote{See Katzarkov, Kontsevich and Pantev \cite{kkp} for a summary and
Auroux, Katzarkov and Orlov \cite{ako} for some partial results.}
\begin{equation}D^{b}(\mathbb{P}^{n})\cong FS( (\mathbb{C}^{*})^{n},W)  \nonumber \end{equation}
between the derived category of $\mathbb{P}^{n}$ and the Fukaya-Seidel category \cite{seidel} of Lefschetz fibration $W$. We denote a Lefschetz thimble of $W$ to be $\Delta^{n}$ which is diffeomorphic to a $n$-dimensional disk.

The mirror of $K_{\mathbb{P}^{n}}$ \cite{horiiqbalvafa} is the hypersurface
\begin{equation}\check{X}=\{(x,y)\in(\mathbb{C}^{*})^{n}\times \mathbb{C}^{2}\textrm{ } | \textrm{ } y_{1}y_{2}+W(x)=z \},  \nonumber \end{equation}  where $z$ is a regular value of $W$. In \cite{seidel2}, Seidel introduced the suspension of Lefschetz fibrations and interpreted $\check{X}$ as
the double suspension of a regular fiber of $W$. Under the double suspension, $\partial(\Delta^{n})$ becomes a Lagrangian sphere $L$ ($\cong \mathbb{S}^{n+1}$) in $\check{X}$. Then one obtains the mirror analog of (\ref{equ 2})
\begin{equation}HF^{*}_{\check{X}}(L,L)\cong HF^{*}_{(\mathbb{C}^{*})^{n}}(\Delta^{n},\Delta^{n})\oplus HF^{n+1-*}_{(\mathbb{C}^{*})^{n}}(\Delta^{n},\Delta^{n})^{*}, \nonumber \end{equation}
where $HF^{*}_{(\mathbb{C}^{*})^{n}}(\Delta^{n},\Delta^{n})\triangleq H^{*}(\Delta^{n},\mathbb{Z})$.

\subsection{Cyclic $A_{\infty}$-algebras in Lagrangian Floer theory}
We recall definitions of cyclic $A_{\infty}$-algebras over a field $\mathbb{K}$ and their existences on Lagrangian Floer cohomologies which
are needed for the completion of a proof of Theorem \ref{A-side local Darboux thm}.
\begin{definition}(\cite{fukaya})
An $A_{\infty}$-algebra is a $\mathbb{Z}$-graded $\mathbb{K}$-vector space
\begin{equation}A=\bigoplus_{p\in\mathbb{Z}} A^{p}    \nonumber\end{equation}
endowed with graded maps
\begin{equation}m_{n}: A^{\otimes n}\rightarrow A, n\geq0    \nonumber\end{equation}
of degree $2-n$ such that for any $k\geq0$, we have
\begin{equation}\sum_{k_{1}+k_{2}=k+1}\sum_{i}(-1)^{deg(x_{1})+\cdot\cdot\cdot+deg( x_{i-1})+i-1}m_{k_{1}}(x_{1},\cdot\cdot\cdot,m_{k_{2}}(x_{i},\cdot\cdot\cdot,x_{i+k_{2}-1}),\cdot\cdot\cdot,x_{k}) )=0.   \nonumber\end{equation}
\end{definition}
As we do not require $m_{1}^{2}=0$, it is sometimes called curved $A_{\infty}$-algebra \cite{kelley}. Following \cite{fooo}, we call an $A_{\infty}$-algebra strict if $m_{0}=0$, in which case we have $m_{1}^{2}=0$. To reflect the Calabi-Yau $n$-algebra structure, we introduce the cyclic condition on $A_{\infty}$-algebras.
\begin{definition}\label{def of cyclic str}(\cite{fukaya})
A finite dimensional $A_{\infty}$-algebra $(A,m_{k})$ is called $n$-cyclic, if there exists a homogenous bilinear map
\begin{equation}Q: A\otimes A\rightarrow\mathbb{K}[-n]  \nonumber\end{equation}
such that

$\bullet$ $Q(x,y)=(-1)^{(deg x+1)(deg y+1)+1}Q(y,x)$, \\

$\bullet$ $Q(m_{k}(x_{1},...,x_{k}),x_{0})=(-1)^{*}Q(m_{k}(x_{0},...,x_{k-1}),x_{k})$, \\
${}$ \\
where $*=(deg(x_{0})+1)(deg(x_{1})+\cdot\cdot\cdot+deg(x_{k})+k)$.

\end{definition}
A typical example of strict $n$-cyclic $A_{\infty}$-algebra is the extension group of sheaves on compact Calabi-Yau $n$-folds \cite{poli}, \cite{ks1}, \cite{tu}. The mirror analog in Lagrangian Floer theory is due to Fukaya \cite{fukaya} and Fukaya, Oh, Ohta and Ono \cite{fooo}.

We take a relatively spin compact Lagrangian submanifold $L$ in a compact symplectic manifold $X$. The universal Novikov ring is
\begin{equation}\Lambda_{0,nov}=\bigg\{\sum_{i=0}^{\infty}a_{i}T^{\lambda_{i}}e^{n_{i}}\mid a_{i}\in \mathbb{Q}, \lambda_{i}\in\mathbb{R}_{\geq0},n_{i}\in\mathbb{Z}\textrm{ }  \textrm{and} \textrm{ } \lim_{i\rightarrow\infty}\lambda_{i}=\infty \bigg\}, \nonumber\end{equation}
with maximal ideal $\Lambda_{0,nov}^{+}$ which consists of elements such that $\lambda_{i}\in\mathbb{R}_{>0}$.
If $L$ has zero Maslov index and $X$ is Calabi-Yau, $H^{*}(L;\Lambda_{0,nov}^{+})=H^{*}(L;\mathbb{Q})\otimes_{\mathbb{Q}}\Lambda_{0,nov}^{+}$ will have a $\mathbb{Z}$-graded cyclic $A_{\infty}$-algebra structure, i.e.
\begin{theorem}\label{cyclic A inf str}(Fukaya \cite{fukaya}, Fukaya-Oh-Ohta-Ono \cite{fooo})${}$ \\
Let $L$ be a compact relatively spin Lagrangian submanifold of zero Maslov index inside a Calabi-Yau $n$-fold $X$ \footnote{It is compact or convex at infinity with $c_{1}(X)=0$.}. Then $H^{*}(L;\Lambda_{0,nov}^{+})$ has a $n$-cyclic $A_{\infty}$-algebra structure with respect to the Poincar\'{e} pairing, which is well-defined up to isomorphisms.
\end{theorem}
Finally, by combining Theorem \ref{cyclic A inf str} and Lemma \ref{lemma on cyclic str}, we finish the proof of Theorem \ref{A-side local Darboux thm}.

\section{Appendix on the B-model story$-DT_{4}$ theory}
We recall some basic facts in Donaldson-Thomas theory on Calabi-Yau 4-folds. The main references are Borisov-Joyce's article \cite{bj} and the authors's preprints \cite{cao, caoleung, caoleung3, caoleung2}.

We start with a compact Calabi-Yau 4-fold $(X,\mathcal{O}_{X}(1))$ ($Hol(X)=SU(4)$) with a Ricci-flat K\"ahler metric $g$ \cite{yau}, a K\"ahler form $\omega$, a holomorphic four-form $\Omega$, and a topological bundle with a Hermitian metric $(E,h)$. We define
\begin{equation}*_{4}=(\Omega\lrcorner)\circ*: \Omega^{0,2}(X,EndE)\rightarrow \Omega^{0,2}(X,EndE),
\nonumber \end{equation}
with $*_{4}^{2}=1$ and it splits the corresponding harmonic subspace into (anti-)self-dual parts.

The $DT_{4}$ equations are defined to be
\begin{equation}\label{DT4 equations} \left\{ \begin{array}{l}
  F^{0,2}_{+}=0 \\ F\wedge\omega^{3}=0 ,    
\end{array}\right. \end{equation}
where the first equation is $F^{0,2}+*_{4}F^{0,2}=0$ and we assume $c_{1}(E)=0$ for simplicity in the moment map equation $ F\wedge\omega^{3}=0$.

We denote $\mathcal{M}^{DT_{4}}(X,g,[\omega],c,h)$ or simply $\mathcal{M}^{DT_{4}}_{c}$ to be the space of
gauge equivalence classes of solutions to the $DT_{4}$ equations on $E$ (with $ch(E)=c$).

We take $\mathcal{M}_{c}^{bdl}$ to be the moduli space of slope-stable holomorphic bundles with fixed Chern character $c$.
By Donaldson-Uhlenbeck-Yau's theorem \cite{UY}, we can identify it with the moduli space of gauge equivalence classes of
solutions to the holomorphic HYM equations
\begin{equation}\label{hym} \left\{ \begin{array}{l}
  F^{0,2}=0 \\ F\wedge\omega^{3}=0.    
\end{array}\right.\end{equation}
By Lemma 4.1 \cite{caoleung}, if $ch_{2}(E)\in H^{2,2}(X,\mathbb{C})$,
then $F^{0,2}_{+}=0\Rightarrow F^{0,2}=0$. In particular, if $\mathcal{M}_{c}^{bdl}\neq\emptyset$,
then $\mathcal{M}_{c}^{DT_{4}}\cong\mathcal{M}_{c}^{bdl}$ as sets.
The comparison of analytic structures is given by
\begin{theorem}\label{mo mDT4}(Theorem 1.1 \cite{caoleung}) We assume $\mathcal{M}_{c}^{bdl}\neq\emptyset$ and
fix $d_{A}\in\mathcal{M}_{c}^{DT_{4}}$, then  \\
(1) there exists a Kuranishi map $\tilde{\tilde{\kappa}}$ of $\mathcal{M}_{c}^{bdl}$ at $\overline{\partial}_{A}$
(the (0,1) part of $d_{A}$) such that $\tilde{\tilde{\kappa}}_{+}$ is a Kuranishi map of $\mathcal{M}_{c}^{DT_{4}}$ at $d_{A}$, where
\begin{equation} \xymatrix@1{
\tilde{\tilde{\kappa}}_{+}=\pi_{+}(\tilde{\tilde{\kappa}}): H^{0,1}(X,EndE) \ar[r]^{\quad \quad \quad \tilde{\tilde{\kappa}}}
& H^{0,2}(X,EndE)\ar[r]^{\pi_{+}} & H^{0,2}_{+}(X,EndE) }  \nonumber \end{equation}
and $\pi_{+}$ is projection to the self-dual forms; \\
(2) the closed imbedding between analytic spaces possibly with non-reduced structures $\mathcal{M}_{c}^{bdl}\hookrightarrow \mathcal{M}_{c}^{DT_{4}}$
is also a homeomorphism between topological spaces.
\end{theorem}
\begin{remark}
By Proposition 10.10 \cite{caoleung}, the map $\tilde{\tilde{\kappa}}$ satisfies $Q_{Serre}(\tilde{\tilde{\kappa}},\tilde{\tilde{\kappa}})\geq0$,
where $Q_{Serre}$ is the Serre duality pairing on $H^{0,2}(X,EndE)$.
\end{remark}
To define Donaldson type invariants using $\mathcal{M}_{c}^{DT_{4}}$, we need to give it a good compactification (such that it carries
a deformation invariant fundamental class). For this purpose, we introduce $\mathcal{M}_{c}(X,\mathcal{O}_{X}(1))$ or
simply $\mathcal{M}_{c}$ to be the Gieseker moduli space of $\mathcal{O}_{X}(1)$-stable sheaves on $X$ with given Chern character $c$.
Motivated by Theorem \ref{mo mDT4}, we make the following definition.
\begin{definition}\label{gene DT4 moduli}(\cite{caoleung})
We call a $C^{\infty}$-scheme, $\overline{\mathcal{M}}^{DT_{4}}_{c}$ generalized $DT_{4}$ moduli space if there exists a homeomorphism
\begin{equation}\mathcal{M}_{c}\rightarrow\overline{\mathcal{M}}^{DT_{4}}_{c}  \nonumber \end{equation}
such that at each closed point of $\mathcal{M}_{c}$, say $\mathcal{F}$, $\overline{\mathcal{M}}^{DT_{4}}_{c}$ is locally isomorphic
to $\kappa_{+}^{-1}(0)$, where
\begin{equation}\kappa_{+}=\pi_{+}\circ\kappa:Ext^{1}(\mathcal{F},\mathcal{F})\rightarrow Ext^{2}_{+}(\mathcal{F},\mathcal{F}),
\nonumber \end{equation}
$\kappa$ is a Kuranishi map at $\mathcal{F}$
and $Ext^{2}_{+}(\mathcal{F},\mathcal{F})$ is a half dimensional real subspace of $Ext^{2}(\mathcal{F},\mathcal{F})$ on which the
Serre duality quadratic form is real and positive definite.
\end{definition}
\begin{remark}${}$ \\
1. The existence of generalized $DT_{4}$ moduli spaces is proved by Borisov-Joyce \cite{bj}. The authors proved their existence
as real analytic spaces in certain cases and defined the corresponding virtual fundamental classes \cite{cao},\cite{caoleung}. \\
2. For fixed data $(X,\mathcal{O}_{X}(1),c)$, the generalized $DT_{4}$ moduli space may not be unique. However,
they all carry the same virtual fundamental classes.
\end{remark}
The proof of Borisov-Joyce's gluing result is divided into two parts. Firstly, they used good local models of $\mathcal{M}_{c}$, i.e.
local 'Darboux charts' in the sense of Brav, Bussi and Joyce \cite{bbj}. Then they choosed the half dimensional real subspace
$Ext^{2}_{+}(\mathcal{F},\mathcal{F})$ appropriately and used partition of unity and homotopical algebra to glue $\kappa_{+}$.
We state an analytic version of the local 'Darboux charts' and give a proof using gauge theory.
\begin{theorem}\label{B-side local Darboux thm}(Brav-Bussi-Joyce \cite{bbj} Corollary 5.20, see also Theorem 10.7 \cite{caoleung})  ${}$ \\
Let $\mathcal{M}_{c}$ be a Gieseker moduli space of stable sheaves on a compact $CY_{4}$.
For any closed point $\mathcal{F}\in\mathcal{M}_{c}$, there exists an analytic neighborhood $U_{\mathcal{F}}\subseteq\mathcal{M}_{c}$,
a holomorphic map near the origin
\begin{equation}\kappa: Ext^{1}(\mathcal{F},\mathcal{F})\rightarrow Ext^{2}(\mathcal{F},\mathcal{F})   \nonumber\end{equation}
such that $Q_{Serre}(\kappa,\kappa)=0$ and $\kappa^{-1}(0)\cong U_{\mathcal{F}}$ as complex analytic spaces possibly with non-reduced structures,
where $Q_{Serre}$ is the Serre duality pairing on $Ext^{2}(\mathcal{F},\mathcal{F})$.
\end{theorem}
\begin{proof}(Proof of Theorem 10.7 \cite{caoleung})
We use Seidel-Thomas twists \cite{js},\cite{st} transfer the problem to a problem on moduli spaces of holomorphic bundles and then notice that
\begin{equation}
\int Tr(F^{0,2}\wedge F^{0,2})\wedge\Omega=-8\pi^{2}\int ch_{2}(E)\wedge\Omega=0,  \nonumber \end{equation}
as $E$ is holomorphic.
\end{proof}
The above theorem has an application to the unobstructedness of Gieseker moduli spaces.
\begin{corollary}\label{unobs in B side}(Corollary 10.9 \cite{caoleung})
If for any closed point $\mathcal{F}\in\mathcal{M}_{c}$, $dim Ext^{2}(\mathcal{F},\mathcal{F})\leq 1$, then $\mathcal{M}_{c}$ is smooth, i.e. all Kuranishi maps are zero.
\end{corollary}

\end{document}